\documentclass[10pt]{amsart}
\usepackage{enumerate,amsmath,amssymb,latexsym,
amsfonts, amsthm, amscd, MnSymbol}

\theoremstyle{plain}

\newtheorem{theorem}{Theorem}

\numberwithin{equation}{section}

\addtocounter{section}{-1}

\begin{document}

\title {Beyond Quartic Hamiltonians}

\date{}

\author[P.L. Robinson]{P.L. Robinson}

\address{Department of Mathematics \\ University of Florida \\ Gainesville FL 32611  USA }

\email[]{paulr@ufl.edu}

 \keywords{}

\begin{abstract}

We investigate the Hamiltonian system generated by a homogeneous binary polynomial $U$ of order greater than four.  In particular, we study the circumstances under which the associated Hessian is a Weierstrass $\wp$ function and find that the vanishing of two covariants of $U$ is involved. 

\end{abstract}

\maketitle

\medbreak

\section*{Introduction} 

In [5] we showed that along the Hamilton curves generated by a homogeneous cubic or quartic polynomial in the plane, a suitable rescaling of the associated Hessian satisfies the first-order differential equation 
$$\overset{\circ}{F} \,^2 = 4 F^3 - g_2 F - g_3$$
characteristic of a Weierstrass $\wp$ function, where $g_2$ and $g_3$ are constants of the motion. In the present paper we consider as Hamiltonian a homogeneous binary polynomial of higher order and find that the situation is a little more complicated. Along each Hamilton curve, the associated Hessian again satisfies a differential equation of the same form, but now $g_2$ and $g_3$ need not be constants of the motion: their derivatives are given by constant multiples of two covariants of the homogeneous polynomial; a syzygy relating these covariants to the Hamiltonian and its Hessian then shows that (unless the Hessian itself is constant) constancy of $g_2$ and $g_3$ forces $g_2$ to vanish, as in the case of a homogeneous cubic. 

\medbreak 

\section*{Hamiltonians and Hessians}

\medbreak 

We begin by fixing our conventions with regard to Hamiltonians. As a setting for our study, we take the symplectic plane with linear symplectic coordinates $(p, q)$. Given two functions $W$ and $Z$ in the plane, their classical Poisson bracket $\{ W, Z \}$ is their Jacobian $(W, Z)$ with sign reversed: thus 
$$\{ W, Z \} := \frac{\partial W}{\partial q} \frac{\partial Z}{\partial p} -  \frac{\partial W}{\partial p} \frac{\partial Z}{\partial q}$$
coincides with 
\[(Z, W) := 
\begin{vmatrix}
 \frac{\partial Z}{\partial p} & \frac{\partial Z}{\partial q} \\
 \frac{\partial W}{\partial p} & \frac{\partial W}{\partial q}
\end{vmatrix}. 
\]
When the function $U$ is taken as Hamiltonian, the classical Hamilton equations of motion read 
$$\overset{\circ}{q} = \frac{\partial U}{\partial p} \; , \; \; - \overset{\circ}{p} = \frac{\partial U}{\partial q};$$
by a {\it Hamilton curve} we shall mean a solution curve to this system. The derivative of any function $V$ along such a Hamilton curve $\gamma$ is given by 
$$(V \circ \gamma)^{\circ} = \{ V, U \} \circ \gamma = (U, V) \circ \gamma.$$
We shall often write more simply 
$$\overset{\circ}{V} = \{ V, U \} = (U, V) =  U_p V_q - U_q V_p$$
where the derivative $^{\circ}$ is taken along a Hamilton curve, which has been suppressed from the notation. 

\medbreak 

Throughout, our planar Hamiltonian will be a homogeneous polynomial: a (binary) {\it quantic} in classical terminology. The general theory of covariants and invariants for quantics was developed largely by Cayley and Sylvester in the nineteenth century; in particular, Cayley produced a sequence of ten memoirs on the topic between 1854 and 1878. Salmon [6] offers an almost contemporaneous account of this classical theory; Elliott [2] is a slightly more recent elaboration, still in the classical sprit, as is Grace and Young [3]. As we shall encounter several covariants of our quantic, we prepare the analysis by fixing our conventions regarding them.  

\medbreak 

Thus, let $U$ be the (binary) quantic of order $N$ given by 
$$U = a_0 p^N + N a_1 p^{N - 1} q + \frac{1}{2} N (N - 1) a_2 p^{N - 2} q^2 + \cdots + a_N q^N$$
where the inclusion of binomial coefficients is both traditional and simplifying. In short, 
$$U = \sum_{n = 0}^{N}  \binom{N}{n} \: a_n \: p^{N - n} q^n$$
or
$$U = (a_0, \dots , a_N) (p, q)^N$$
in notation introduced by Cayley. Although much of what follows applies quite generally, we shall suppose that $N > 4$ unless otherwise stated. 

\medbreak 

Perhaps the most important covariant associated to $U$ is its Hessian: the determinant of its matrix of second partial derivatives, thus 
\[
\begin{vmatrix}
 U_{pp} & U_{pq} \\
 U_{qp} & U_{qq}
\end{vmatrix} .
\]
Each differentiation brings down the order; we clear factors to normalize the Hessian and write $H$ for its normalized form: thus 
\[
\begin{vmatrix}
 U_{pp} & U_{pq} \\
 U_{qp} & U_{qq}
\end{vmatrix} = N^2 (N - 1)^2 H
\]
where 
$$H = (a_0 a_2 - a_1^2) p^{2 N - 4} + (N - 2) (a_0 a_3 - a_1 a_2) p^{2 N - 5} q + \dots \; .$$

\medbreak 

The Jacobian $(U, H)$ of $U$ and $H$ is a second important covariant of $U$. Differentiation again introduces factors, which we remove to define the normalized covariant $G$: thus 
\[
\begin{vmatrix}
 U_p & U_q \\
 H_p & H_q
\end{vmatrix} = N (N - 2) G
\]
where 
$$G = (a_0^2 a_3 - 3 a_0 a_1 a_2 + 2 a_1^3) p^{3 N - 6} + \dots \; .$$

\medbreak 

Two further covariants associated to $U$ arise from its quartic emanant; briefly, the details are as follows. Introduce auxiliary variables $(P, Q)$: the fourth emanant $(P \partial_p + Q \partial_q)^4 U$ of $U$ is a quartic polynomial 
$$A P^4 + 4 B P^3 Q + 6 C P^2 Q^2 + 4 D P Q^3 + E Q^4$$
in $(P, Q)$ with coefficients that are polynomials in $(p, q)$. This quartic has familiar  invariants 
$$AE - 4 BD + 3 C^2$$
and 
$$ACE + 2 BCD - AD^2 - B^2 E - C^3$$
which are then covariants of $U$. Again we normalize, defining the covariants $S$ and $T$ of $U$ by 
$$AE - 4 BD + 3 C^2 = [N(N - 1)(N - 2)(N - 3)]^2 S$$
and 
$$ACE + 2 BCD - AD^2 - B^2 E - C^3 = [N(N - 1)(N - 2)(N - 3)]^3 T$$
where 
$$S = (a_0 a_4 - 4 a_1 a_3 + 3 a_2^2) p^{2 N - 8} + \dots$$
and 
$$T = (a_0 a_2 a_4 + 2 a_1 a_2 a_3 - a_0 a_3^2 - a_1^2 a_4 - a_2^3) p^{3 N - 12} + \dots \; .$$

\medbreak 

We remark that each covariant of $U$ may be recovered from its source or leader, according to a theorem of Roberts [4]: thus, $H$ may be recovered from its leading coefficient $a_0 a_2 - a_1^2$ while $G$ may be recovered from its leading coefficient $a_0^2 a_3 - 3 a_0 a_1 a_2 + 2 a_1^3$; likewise for $S$ and $T$. Discussions of this point and its consequences may be found in [2] and [6]. 

\medbreak 

These (normalized) covariants $H, G, S$ and $T$ are related to $U$ by the following syzygy. 

\medbreak 

\begin{theorem} \label{syz1}
$G^2 + 4 H^3 + U^3 T = U^2 S H.$
\end{theorem} 

\begin{proof} 
Classical: the aforementioned theorem of Roberts [4] implies that we need only check correct behaviour of sources; this check is entirely straightforward. 
\end{proof} 

\medbreak 

The differential equation characteristic of Weierstrass $\wp$ functions is now on the verge of manifestation. To be explicit, by a {\it Weierstrass equation} we shall mean a first-order ordinary differential equation of the form 
\begin{equation} \label{DE}
\overset{\circ}{\Phi}\,^2 = 4 \Phi^3 - g_2 \Phi - g_3 \tag{{\bf DE}}
\end{equation}
wherein $g_2$ and $g_3$ are functions; we shall call this Weierstrass equation {\it proper} when $g_2$ and $g_3$ are constant. In the proper case, introduce the discriminant $\Delta = g_2^3 - 27 g_3^2$. When $\Delta$ is zero, \ref{DE} has elementary solutions, the cubic on its right side having a repeated root. When $\Delta$ is nonzero, the nonconstant solutions of \ref{DE} are Weierstrass $\wp$ functions. 
\medbreak 

Now, a Weierstrass equation emerges from a simple rescaling of the Hessian. 

\medbreak 

\begin{theorem} \label{Phi}
The scalar multiple 
$$\Phi = -  [N(N - 2)]^2 H$$ 
of the Hessian satisfies the first-order differential equation 
$$\overset{\circ}{\Phi} \, ^2 = 4 \Phi^3 - N^4 (N - 2)^4 U^2 S \Phi - N^6 (N - 2)^6 U^3 T$$
along each Hamilton curve of $U$. 
\end{theorem} 
\begin{proof}
As usual, we suppress notation for the Hamilton curve, differentiation of $\Phi$ along which results in 
$$\overset{\circ}{\Phi} = -  [N(N - 2)]^2 \overset{\circ}{H} = -  [N(N - 2)]^2 (U, H) = -  [N(N - 2)]^3 G$$
whence we deduce that 
$$\overset{\circ}{\Phi}\,^2 = [N(N - 2)]^6 G^2.$$
Substitution from Theorem \ref{syz1} and rearrangement conclude the argument. 
\end{proof} 

\medbreak 

This differential equation has the Weierstrass form \ref{DE}: 
$$\overset{\circ}{\Phi}\,^2 = 4 \Phi^3 - g_2 \Phi - g_3$$
where 
$$g_2 = [N^2 (N - 2)^2 U]^2 S$$
and 
$$g_3 = [N^2 (N - 2)^2 U]^3 T.$$
Here, the covariants $\Phi$, $g_2$ and $g_3$ are evaluated along the Hamilton curve; in general, $g_2$ and $g_3$ need not be constant. 

\medbreak 

Our primary concern is with those situations in which this Weierstrass equation is proper: those situations in which $g_2$ and $g_3$ are {\it constant} along the Hamilton curve. Of course, $U$ itself is constant along each Hamilton curve: indeed, 
$$\overset{\circ}{U} = (U, U) = 0.$$ 
If this constant value of $U$ is zero, then $g_2 = g_3 = 0$ and the differential equation \ref{DE} simplifies, its nonzero solutions being inverse squares. Discounting this case in which $U$ is zero, we are interested in a Hamilton curve along which the covariants $S$ and $T$ are constant. 

\medbreak 

Naturally, we detect constancy along the Hamilton curve by differentiation. Entirely routine computations reveal that along each Hamilton curve, 
$$\overset{\circ}{S} = (U, S) = N (N - 4) [ S_0 p^{3 N - 10} + \cdots \; ]$$
and 
$$\overset{\circ}{T} = (U, T) = N (N - 4) [ T_0 p^{4 N - 14} + \cdots \; ]$$
where 
$$S_0 = a_0^2 a_5 - 5 a_0 a_1 a_4 + 2 a_0 a_2 a_3 + 8 a_1^2 a_3 - 6 a_1 a_2^2$$
and 
$$T_0 = a_0^2 a_2 a_5 - a_0^2 a_3 a_4 - a_0 a_1^2 a_5 - 2 a_0 a_1 a_2 a_4 + 4 a_0 a_1 a_3^2 - a_0 a_2^2 a_3 + 3 a_1^3 a_4 - 6 a_1^2 a_2 a_3 + 3 a_1 a_2^3.$$
Accordingly, we focus our attention on the Jacobian covariants  $(U, S)$ and $(U, T)$. 

\medbreak 

A useful alternative expression for $(U, T)$ is as follows. 

\medbreak 

\begin{theorem} \label{switch}
$2 (N - 2) (U, T) = N (H, S).$
\end{theorem} 

\begin{proof} 
By direct calculation of the Jacobian determinants; once again, the work is simplified by the circumstance that only sources need be checked. 
\end{proof} 

\medbreak 

Any three (binary) quantics are related by an elementary syzygy that is now perhaps less well known than it should be. 

\medbreak 

\begin{theorem} \label{syz}
If $X, Y, Z$ are quantics of orders $\ell, m, n$ then 
$$\ell X (Y, Z) + m Y (Z, X) + n Z (X, Y) = 0.$$
\end{theorem} 

\begin{proof} 
By virtue of the Euler theorem on homogeneous functions, $p Z_p + q Z_q = n Z$ with similar equations for $X$ and $Y$. It follows that 
$$n Z (X, Y) = (p Z_p + q Z_q)(X_p Y_q - X_q Y_p)$$
with similar expressions for $\ell X (Y, Z)$ and $m Y (Z, X)$. Summing,  
$$\ell X (Y, Z) + m Y (Z, X) + n Z (X, Y) = \lambda p + \mu q$$
where 
\[\lambda = 
\begin{vmatrix}
 X_p & X_p & X_q \\
 Y_p & Y_p & Y_q \\
Z_p & Z_p & Z_q
\end{vmatrix} = 0
\]
and 
\[\mu = 
\begin{vmatrix}
 X_q & X_p & X_q \\
 Y_q & Y_p & Y_q \\
Z_q & Z_p & Z_q
\end{vmatrix} = 0.
\]
\end{proof} 

\medbreak 

We note here that Grace and Young [3] normalize the Jacobian as a transvectant: in Section 77, their Jacobian of $X$ and $Y$ is ours divided by $\ell m$; with their normalization, Theorem \ref{syz} simply asserts that 
$$X (Y, Z) + Y (Z, X) + Z (X, Y) = 0.$$

\medbreak 

A special consequence of this quite general result is the following syzygy between covariants of $U$. 

\medbreak 

\begin{theorem} \label{syz2}
$(N - 4) \: (U, H) \,S = (N - 2) \: [(U, S)\,  H - (U, T) \, U ].$
\end{theorem} 

\begin{proof} 
Invoke Theorem \ref{syz} in case $X = H$ (of order $2 N - 4$), $Y = S$ (of order $2 N - 8$), $Z = U$ (of order $N$); apply Theorem \ref{switch} to replace $N (H, S)$ by $2 (N - 2) (U, T)$ and rearrange. 
\end{proof} 

\medbreak 

Note that if these covariants are evaluated along a Hamilton curve $\gamma$ then
$$(N - 4) \overset{\circ}{H} S = (N - 2) [\overset{\circ}{S} H - \overset{\circ}{T} U]$$
on account of the circumstance that differentiation along $\gamma$ is given by $\overset{\circ}{V} = (U, V)$. 

\medbreak 

We may now return to Theorem \ref{Phi} and the Weierstrass equation \ref{DE} for the rescaled Hessian $\Phi$: thus, 
$$\overset{\circ}{\Phi}\,^2 = 4 \Phi^3 - g_2 \Phi - g_3$$
where $g_2 = [N^2 (N - 2)^2 U]^2 S$ and  $g_3 = [N^2 (N - 2)^2 U]^3 T$. Recall that the differential equation is run along a Hamilton curve $\gamma$ and that we suppose the constant value of $U$ to be nonzero. This Weierstrass equation is proper precisely when $g_2$ and $g_3$ are constant: thus, precisely when $S$ and $T$ are constant; so, precisely when their derivatives $(U, S)$ and $(U, T)$ are constantly zero. Here, constancy is of course along the Hamilton curve $\gamma$. 

\medbreak 

At this point, the syzygy of Theorem \ref{syz2} exerts its influence. As $N > 4$, the simultaneous vanishing of $(U, S)$ and $(U, T)$ forces the product $(U, H) S$ to vanish along $\gamma$. If the constant value of $S$ is nonzero, then the derivative $(U, H)$ vanishes along $\gamma$ so that $\Phi$ is constant, its value $\Phi_0$ satisfying the cubic $4 \Phi_0^3 = g_2 \Phi_0 + g_3$. Consequently, if the (proper) Weierstrass equation considered here has a nonconstant solution then the constant value of $S$ (and hence of $g_2$)  is zero. In this connexion, it is interesting to note that $g_2$ is zero for cubic Hamiltonians but can be nonzero for quartic Hamiltonians; see [5] for this. 

\medbreak 

These findings may be summarized as follows. The Weierstrass equation of Theorem \ref{Phi} is proper along the Hamiltonian curve $\gamma$ precisely when the covariants $(U, S)$ and $(U, T)$ vanish along $\gamma$. In such a case, if the (rescaled) Hessian $\Phi$ is nonconstant along $\gamma$ then the constant value of $S$ (hence of $g_2$) is zero; moreover, a nonconstant $\Phi$ will be a Weierstrass $\wp$ function when the constant value of $T$ is nonzero. 

\medbreak 

As a special case, if we insist that the covariants $(U, S)$ and $(U, T)$ themselves vanish as quantics, then the Weierstrass equation of Theorem 2 is proper along each Hamilton curve; in this special case, only along those Hamilton curves that originate (and hence remain) in the zero-set of $S$ can $\Phi$ be a Weierstrass $\wp$ function. 

\medbreak 

In general, given the quantic $U$ as Hamiltonian, a Hamilton curve $\gamma$ is determined by its initial point $\gamma_0$; the question whether or not the covariants $(U, S)$ and $(U, T)$ are zero along $\gamma$ is therefore decided by $\gamma_0$. The task of elucidating this matter is reserved for a future publication. 

\medbreak 

The syzygy in Theorem \ref{syz2} is only one of many involving the covariants associated to a quantic; the influence of other syzygies is also reserved for a subsequent study. 

\medbreak 

We should comment on the origin of the Hessian $H$ within the Hamiltonian theory for the quantic $U$. Along a Hamilton curve, 
$$\overset{\circ}{q} = \frac{\partial U}{\partial p} \; , \; \; - \overset{\circ}{p} = \frac{\partial U}{\partial q}$$
whence a further differentiation produces 
$$\overset{\circ \circ}{p} = U_q U_{q p} - U_{q q} U_p \;, \; \; \overset{\circ \circ}{q} = U_p U_{p q} - U_{p p} U_q.$$
The Euler equation applied to the homogeneous functions $U_p$ and $U_q$ yields 
$$p U_{pp} + q U_{pq} = (N - 1)  U_p \; , \; \; p U_{qp} + q U_{qq}= (N - 1)  U_q.$$
Isolation of $p$ and $q$ separately results in 
$$(N - 1) [U_q U_{q p} - U_{q q} U_p] = - [U_{pp} U_{qq} - U_{pq} U_{qp}]\,p$$
and 
$$(N - 1) [U_p U_{p q} - U_{p p} U_q] = -[U_{pp} U_{qq} - U_{pq} U_{qp}]\, q.$$
Expressing our conclusion in terms of the normalized Hessian, 
$$\overset{\circ \circ}{p} = - [N^2 (N - 1) H] \, p$$
and 
$$\overset{\circ \circ}{q} = - [N^2 (N - 1) H] \, q.$$
Stripping the coordinates, 
$$\overset{\circ \circ}{\gamma} = - [N^2 (N - 1) H] \gamma.$$

\medbreak 

In terms of the rescaled Hessian $\Phi$ this differential equation assumes the form 
$$\overset{\circ \circ}{\gamma} = \Big[\frac{N - 1}{(N - 2)^2} \: \Phi \Big] \gamma.$$
In particular, when $\Phi$ is a Weierstrass $\wp$ function, this is a Lam\'e equation 
$$\overset{\circ \circ}{\gamma} = n (n + 1) \, \wp \gamma$$
with parameter $n = 1/(N - 2)$ that is not integral (unless $N = 3$). 

\medbreak 

We should perhaps also comment on the origin of this paper. It began as a continuation of [5] devoted solely to the case of homogeneous quintic polynomials, with only [6] as a reference. In this context, we found the syzygy of Theorem \ref{syz1} by hand, not by consulting [6]; indeed, this syzygy for quintics does not find its way explicitly into [6]. We then turned to the work of Cayley: this syzygy appears in his eighth memoir and is listed in Table 89 of his ninth memoir [1] as $A^3 D - A^2 B C + 4 C^3 + F^2 = 0$; in fact, it appeared previously on page 173 of [4] but with a sign change in the Hessian. The syzygy of Theorem \ref{syz2} for quintics already occurs in the second memoir of Cayley; it is listed in Table 89 of [1] as $A I + B F - C E = 0$. Our proof of this syzygy follows a suggestion in Section 194 of [6].

\bigbreak

\begin{center} 
{\small R}{\footnotesize EFERENCES}
\end{center} 
\medbreak

[1] A. Cayley, {\it A Ninth Memoir on Quantics}, Philosophical Transactions of the Royal Society of London, {\bf 161} 17-50 (1871).

\medbreak 

[2] E.B. Elliott, {\it An Introduction to the Algebra of Quantics}, Oxford University Press, Second Edition (1913). 

\medbreak 

[3] J.H. Grace and A. Young, {\it The Algebra of Invariants}, Cambridge University Press (1903). 

\medbreak 

[4] M. Roberts, {\it On the covariants of a binary quantic of the $n^{th}$ degree}, Quarterly Journal of Pure and Applied Mathematics {\bf IV} 168-178 (1861). 

\medbreak 

[5] P.L. Robinson, {\it Cayley-Sylvester invariants and the Hamilton equations}, arXiv 1505.02273 (2015). 

\medbreak 

[6] G. Salmon, {\it Lessons Introductory to the Modern Higher Algebra}, Hodges, Figgis and Company, Fourth Edition (1885). 

\medbreak

\end{document}